\documentclass[10pt,a4paper,reqno]{amsart}
\usepackage{amssymb,amsmath,amsthm}
\usepackage[pdftex]{graphicx}
\usepackage{xypic}
\usepackage{multicol}

\numberwithin{equation}{section}
\newtheorem{thm}{Theorem}[section]
\newtheorem{cor}[thm]{Corollary}
\newtheorem{lem}[thm]{Lemma}
\newtheorem{prop}[thm]{Proposition}

\theoremstyle{definition}
\newtheorem{dfn}[thm]{Definition}

\newcommand{\Z}{\mathbb{Z}}

\newcommand{\Zp}{\mathbb{Z}/p\mathbb{Z}}
\newcommand{\p}{\mathcal{P}^{1}_{\ast}}
\newcommand{\nil}{\textrm{nil}}
\newcommand{\ord}{\textrm{ord}}
\newcommand{\degp}{\underline{p}}
\newcommand{\degpp}{\underline{p}^2}
\newcommand{\degppp}{\underline{p}^3}
\newcommand{\degpm}{\underline{p}^m}

\author{Andrew Russhard}

\title{Power Maps On Quasi-$p$-Regular $SU(n)$}
\begin{document}	

\begin{abstract}
	In the paper we will introduce show that the $p^3$ power map on $SU(p+t-1)$ is an H-map for $2\leq t\leq p-1$. To do this we will consider a fibration whose base space is $SU(p+t-1)$ with the property that there is a section into the total space. We will then use decomposition methods to identify the fibre and the map from it to the total space. This information will be used to deduce information about $SU(p+t-1)$. In doing this we draw together recent work of Kishimoto and Theriault with more classical work of Cohen and Neisendorfer, and make use of the classical theorems of Hilton and Milnor, James and Barrett. 
\end{abstract}

\maketitle

\section{Introduction}

Lie groups play an integral role in many areas of mathematics.
Despite this there are still open problems concerning their multiplicative properties.
In the 1970's, McGibbon, Arkowitz and others looked at the following problem.
\newline
\newline
\indent If $X$ is a connected, homotopy associative H-space which is homotopy equivalent to a finite CW-complex, when is the map $x\mapsto x^{k}$ an H-map?
\newline
\newline
Collectively, they gave the following necessary and sufficient condition for this to be the case \cite{MR0230306,MR587096}.
\newline
\newline
\indent Let $X$ be a connected, homotopy associative, H-space (or localization at an odd prime thereof) then there exists a number $N$ (dependent on $X$) such that the map $x\mapsto x^{k}$ is an H-map iff $k(k-1)\equiv 0$ mod $N$.
\newline
\newline
This condition however depends on a number $N$ which is itself dependent upon the space $X$.
The value of $N$ was calculated for various rank 2 loop spaces as well as for $S^1$, $S^3$ with multiplication induced by the quaternions and $S^7$ with the multiplication induced by the Cayley numbers.
Whilst in principle their work offers a complete solution, the value of $N$ is known in relatively few cases.
What we shall show in this paper is that for the $p$-localization of certain Lie groups $SU(n)$, the value of $N$ is a multiple of $p^2(p^2-1)$ where $p$ is a prime $>5$.
To put this another way, we will show that the $p^2$ power map on the $p$-localization of $SU(n)$ for $p+1\leq n\leq 2p-2$ is an H-map.

Theriault \cite{T} and Kishimoto \cite{MR2506127} have recently proved results of a similar vein.
Theriault \cite{T} showed a similar result for $p$-regular Lie groups, that is when the Lie groups are homotopy equivalent to a product of spheres.
We consider the more technically demanding case when $SU(n)$ is quasi-$p$-regular, that is, that it is homotopy equivalent to a product of spheres and sphere bundles over spheres.
To do this we will use the recent work of Kishimoto in which he has calculated the orders of certain Samelson products in quasi-$p$-regular $SU(n)$.

We will draw these more recent results together with classical results of James, Cohen and Neisendorfer to achieve our results.
In the second section we will state the main results of this paper along with some preliminaries.
The third section will be given over to making explicit a fibration which plays a pivotal role in the proofs of the main results.
In the final two sections we will prove the main results.  

\section{Main Results}

Let $p$ be an odd prime. A Lie group $X$ is said to be quasi-$p$-regular if it is $p$-locally homotopy equivalent to a product of spheres and sphere bundles over spheres.
Oka \cite{MR0258035} showed that for $2\leq t\leq p$ we have the following $p$-local decomposition
\begin{equation*}
	SU(p+t-1)\simeq B_2\times B_3\times\ldots\times B_t\times S^{2t+1}\times S^{2t+3}\times\ldots\times S^{2p-1}
\end{equation*}
where $B_k$ is a 3-cell complex which sits in a fibration
\begin{equation*}
	S^{2k-1}\rightarrow B_k \rightarrow S^{2(k+p)-3},
\end{equation*}
and $H_{\ast}(B_k,\Zp)\cong\Lambda(x_{2k-1},x_{2(k+p)-3})$ where $|x_j|=j$ and the $x_i$ are linked via a Steenrod operation $\p(x_{2(k+p)-3})=x_{2k-1}$.

From this point on, all spaces will be localized at a prime $p>5$, and homology will have $\Zp$ coefficients.
Let $A_k$ be the $(4k+2p-5)$-skeleton of $B_k$ and let
\begin{equation*}
	A=A_2\vee A_3\vee\ldots \vee A_t\vee S^{2t+1}\vee S^{2t+3}\vee\ldots\vee S^{2p-1}.
\end{equation*}
As the inclusion $A_k\hookrightarrow B_k$ induces the inclusion of the generating set in homology for each $2\leq n \leq p$, there is an obvious inclusion $\overline{i}:A\hookrightarrow SU(p+t-1)$.
Taking the adjoint of this map gives a map $i:\Sigma A\hookrightarrow BSU(p+t-1)$.
There is then a fibration
\begin{equation*}
	F\xrightarrow{\nu} \Sigma A \xrightarrow{i} BSU(p+t-1)
\end{equation*}
for each $2\leq t<p$, which defines the space $F$ and the map $\nu$.
Looping this fibration then gives us the fibration
\begin{equation*}
	\Omega F\xrightarrow{\Omega\nu} \Omega \Sigma A \xrightarrow{\Omega i} SU(p+t-1).
\end{equation*}

We will decompose $\Omega F$ and determine the map $\Omega \nu$ as in the following theorem.
Let $X^{(k)}$ be the $k$-fold self smash of $X$.
For ease of notation, if we write $X_1^{(k_1)}\wedge\ldots\wedge X_n^{(k_n)}$, any term with $k_t=0$ is omitted from the smash product.
For example $X_1^{(k_1)}\wedge X_2^{(0)}\wedge X_3^{(k_3)}$ means $X_1^{(k_1)}\wedge X_3^{(k_3)}$.

Observe that each $A_k$ has one or two cells. 
In such cases Cohen and Neisendorfer \cite{MR764588} showed that there exists a fibration
\begin{equation*}
	\Omega R_k\xrightarrow{\Omega f_k} \Omega\Sigma A_k \rightarrow B_k,
\end{equation*}
and a homotopy decomposition $\Omega\Sigma A_k\simeq B_k\times \Omega R_k$.
If $A_k=S^{2k-1}$ then we have that $R_k=S^{4k-3}$, and if $A_k$ has two cells then $R_k$ is a $5$-cell complex.
The map $f_k$ was shown to factor through Whitehead products. 
Let $\overline{f}_k$ be the composite
\begin{equation*}
	\overline{f}_k: R_k\xrightarrow{f_k} \Sigma A_k\hookrightarrow \Sigma A.
\end{equation*}

\begin{thm}\label{a}
	The space $\Omega F$ in the fibration
	\begin{equation*}
		\Omega F\xrightarrow{\Omega \nu} \Omega \Sigma A \xrightarrow{\Omega i} SU(p+t-1)
	\end{equation*}
	can be decomposed as 
	\begin{equation*}
		\Omega F\simeq \prod_{j\in\mathcal{J}}\Omega\Sigma\left( A_2^{(j_2)}\wedge A_3^{(j_3)} \wedge\ldots\wedge A_t^{(j_t)}\wedge S^{2t+1^{(j_{t+1})}}\wedge\ldots\wedge S^{2p-1^{(j_{p})}} \right)\times \prod_{k=2}^{p}\Omega R_k
	\end{equation*}
	for an appropriate index set $\mathcal{J}$, and under this equivalence the map $\Omega\nu$ restricted to $\Omega R_k$ is $\Omega \overline{f}_k$ and $\Omega \nu$ restricted to each other factor of $F$ is 
	\begin{itemize}
 		\item  a looped Whitehead product if the factor is a smash of four or more spaces;
  		\item  a looped Whitehead product if the factor is a smash of three spaces and $2\leq t <\frac{p+1}{2}$;
  		\item  an ``amended'' looped Whitehead product if the factor is a smash of three spaces and $\frac{p+1}{2}\leq t\leq p$;
  		\item  an ``amended'' looped Whitehead product if the factor is a smash of two spaces.
	\end{itemize}
	By the term ``amended'' looped Whitehead product we mean a map of the form $\Omega(\omega-f)$ where $\omega$ is a Whitehead product and $f$ is a map depending on the stable class $\alpha_1$ which we will make precise in the proof. 
\end{thm}

Once this fibration has been constructed, we will use it to prove the main result of this paper.
\begin{thm}\label{b}
	Let $p>5$ and $2\leq t<p$. Then the $p^3$ power map on $SU(p+t-1)$ is an H-map. 
\end{thm}

\section{Samelson Products in $SU(p+t-1)$}

Over the next two sections we will analyse the fibration
\begin{equation}\label{fib1}
	\Omega F \xrightarrow{\Omega\nu} \Omega\Sigma A \xrightarrow{\Omega i} SU(p+t-1)
\end{equation}
To begin, we identify the map $(\Omega i)_{\ast}$.

Recall that by the Bott-Samelson theorem there is an algebra isomorphism, $H_{\ast}(\Omega\Sigma A)\cong T(\tilde{H}_{\ast}(A))$, where $T(\tilde{H}_{\ast}(A))$ is the free tensor algebra generated by $\tilde{H}_{\ast}(A)$.
It is also known that there is an algebra isomorphism
\begin{equation*}
	H_{\ast}(SU(p+t-1))\cong \Lambda(\tilde{H}_{\ast}(A))
\end{equation*}
where $\Lambda(\tilde{H}_{\ast}(A))$ is the free exterior algebra generated by $\tilde{H}_{\ast}(A)$.
As $\overline{i}_{\ast}$ is the adjoint of $i$, we have that the composite
\begin{equation*}
    A\xrightarrow{E} \Omega\Sigma A \xrightarrow{\Omega i} SU(p+t-1)
\end{equation*}
is homotopic to $\overline{i}$.
Since $\overline{i}_{\ast}$ is the inclusion of the generating set, and $(\Omega i)_{\ast}$ is a multiplicative extension of $\overline{i}_{\ast}$, we obtain that $(\Omega i)_{\ast}$ is the abelianisation of the tensor algebra.

Cohen, Moore and Neisendorfer \cite{MR519355} offer a different way to view the fibration (\ref{fib1}) homologically.
Let $L$ be the free Lie algebra generated by $\tilde{H}_{\ast}(A)$.
Then there is the abelianisation map, $a:L\rightarrow L_{ab}$, where $L_{ab}$ is the free abelian Lie algebra generated by $\tilde{H}_{\ast}(A)$.
The kernel of this map is $[L,L]$, the free Lie algebra generated by the brackets in $L$.
Cohen, Moore and Neisendorfer showed that the following diagram commutes
\[ \xymatrix{ H_{\ast}(\Omega F) \ar[r]^-{(\Omega\nu)_{\ast}} \ar[d]^-{\cong} & H_{\ast}(\Omega\Sigma A) \ar[r]^-{(\Omega i)_{\ast}} \ar[d]^-{\cong} & H_{\ast}(SU(p+t-1)) \ar[d]^-{\cong} \\
U[L,L] \ar[r] & UL \ar[r]^-{Ua} & UL_{ab} }\]
where $U$ denotes the universal enveloping algebra operator and $a$ is the abelianisation.
This suggests we may use Samelson products to decompose $\Omega F$ and identify the map $\Omega \nu$.
To do this we will need information about certain Samelson products in $SU(p+t-1)$.

From here on we will write
\begin{equation*}
	A=A_2\vee A_3\vee\ldots \vee A_p = A_2\vee\ldots\vee A_t\vee S^{2t+1}\vee\ldots\vee S^{2p-1}.
\end{equation*}
In this way, $A_k$ will denote $S^{2k-1}$ if $k>t$.

Let $i_k$ be the composite
\begin{equation*}
	i_k:A_k\hookrightarrow A \xrightarrow{i} SU(p+t-1).
\end{equation*}
It is Samelson products $\langle i_k,i_j\rangle$ and $\langle i_k,\langle i_j,i_l\rangle\rangle$ we shall need information about.
Kishimoto \cite{MR2506127} has examined the length two Samelson products $\langle i_k,i_j\rangle$ in detail.
We record his results in the following proposition.
Let $\ord(f)$ denote the order of the map $f$ and $Y_{k,j}$ be the $(2(k+j+2p-2)-1)$-skeleton of $A_k\wedge A_j$.

\begin{prop}\label{sama}
	Let $t+1\leq k,j\leq p$.
	Then $\ord(\langle i_k,i_j\rangle)=p$ if $k+j\geq p+2$.
    Furthermore, if $(k,j)\neq (p,t)$ then we have that 
    \begin{itemize}
      \item $\ord(\langle i_k,i_j\rangle)=p$ if $k+j\geq p+2$, $2\leq k\leq p$ and $t+1\leq j\leq p$,
      \item $\ord(\langle i_k,i_j\rangle|_{Y_{k,j}})=p$ unless $(k,j)=(p,p)$ if $k+j\geq p+2$ and $2\leq k,j\leq t$,
      \item $\ord(\langle i_k,i_j\rangle)=1$ otherwise.
    \end{itemize}
    Furthermore, if $k+j\neq 2p$ then $\langle i_k,i_j\rangle$ can be compressed to a map
    \begin{equation*}
        A_k\wedge A_j \rightarrow S^{2(k+j-p)+1}\hookrightarrow SU(p+t-1).
    \end{equation*}
    If $t\neq p$ then $\langle i_p,i_p\rangle$ can be compressed to a map
    \begin{equation*}
        S^{2p-1}\wedge S^{2p-1} \rightarrow S^{3}\hookrightarrow SU(p+t-1).
    \end{equation*}
    \qed
\end{prop}

This does a large amount of work for us, however it leaves out two important cases.
First we deal with the case $\langle i_k,i_j\rangle$ where $k+j\geq p+2$ and $2\leq k,j\leq t$, but $(k,j)\neq(p,p)$.
For a co-H-space $X$, let $\degp:X\rightarrow X$ be the degree $p$ map.
Then we get the following upper bound.
\begin{prop}\label{samb}
	Let $(k,j)\neq(p,p)$.
	If $k+j\geq p+2$ and $2\leq k,j\leq t$ then
	\begin{itemize}
  		\item $\ord(\langle i_k,i_j\rangle)\leq p^2$ if $k+j\geq p+t$ and
  		\item $\ord(\langle i_k,i_j\rangle)\leq p^3$ if $k+j\leq p+t-1$.
	\end{itemize}
\end{prop}

\begin{proof}
Firstly we note that the inclusion $Y_{k,j}\hookrightarrow A_k\wedge A_j$ is a co-H-map.
Then we know by Proposition \ref{sama} that the composition
\begin{equation*}
	Y_{k,j}\hookrightarrow A_k\wedge A_j \xrightarrow{f} A_k\wedge A_j \xrightarrow{\langle i_k,i_j\rangle} SU(p+t-1)
\end{equation*}
is trivial.
Therefore we get an extension
\begin{equation*}
	\xymatrix{ Y_{k,j}         \ar[d]                 &                                   &           \\
			   A_k\wedge A_j   \ar[r]^-{\degp} \ar[d] & A_k\wedge A_j \ar[r]^-{\langle i_k,i_j\rangle} & SU(p+t-1) \\
			   S^{2(k+j+2p-3)} \ar[urr]_-{f}          &                                                &           }
\end{equation*}
for some map $f$. So if the order of $f$ is $p^t$ then the order of $\langle i_k,i_j\rangle$ is $p^{t+1}$.
By \cite{MR2506127} we know that if $k+j\leq p+t-1$ then $\pi_{2(k+j+2p-3)}(SU(p+t-1))\cong \Z/p^2\Z$ and that if $k+j\geq p+t$ then $\pi_{2(k+j+2p-3)}(SU(p+t-1))\cong \Zp$.
Therefore if $k+j\geq p+t$ then $\langle i_k,i_j\rangle$ has order at most $p^2$ and if $k+j\leq p+t-1$ then $\langle i_k,i_j\rangle$ has order at most $p^3$.
\end{proof}

Now we deal with the case $(k,j)=(p,t)$.
We will obtain an upper bound for the order of the map $\langle i_p,i_t\rangle: S^{2p-1}\wedge A_t\rightarrow SU(p+t-1)$.

\begin{lem}\label{samb2}
	The Samelson product $\langle i_p,i_t\rangle$ has order $\leq p^2$. 
\end{lem}

\begin{proof}
Consider the inclusion $f:S^{2(p+t-1)}\hookrightarrow A_t\wedge S^{2p-1}$ of the bottom cell.
\begin{equation*}
\xymatrix{ S^{2(p+t-1)}       \ar[d]^-{f}                      &           \\
			 A_t\wedge S^{2p-1} \ar[r]^-{\langle i_t,i_p\rangle} & SU(p+t-1) }
\end{equation*}
Kishimoto \cite{MR2506127} tells us that $\pi_{2(p+t-1)}(SU(p+t-1))\cong \Zp$.
Therefore we have that $\langle i_p,i_t\rangle\circ f\circ \degp\simeq \ast$.
As $f$ is the inclusion of the bottom cell, it is a co-H-map, therefore
\begin{equation*}
	\langle i_p,i_t\rangle\circ \degp\circ f\simeq \ast.
\end{equation*}
So we get an extension for some map $g$:
\begin{equation*}
\xymatrix{ S^{2(p+t-1)}       \ar[d]^-{f}                 &                                                     &           \\
			 A_t\wedge S^{2p-1} \ar[r]^-{\degp} \ar[d] & A_t\wedge S^{2p-1} \ar[r]^-{\langle i_t,i_p\rangle} & SU(p+t-1) \\
			 S^{2(2p+t-2)}      \ar[urr]^-{g}          &                                                     &           }
\end{equation*}
By \cite{MR2506127} $\pi_{2(2p+t-2)}(SU(p+t-1))\cong \Zp$.
Thus the order of $\langle i_t,i_p\rangle$ is $\leq p^2$.
\end{proof}

Kishimoto gives more information than just the order of the Samelson products.
He also calculates that they factor through certain spheres in $SU(p+t-1)$, and we will do the same now.

\begin{lem}\label{samfacb}
	Let $t\neq p$. The Samelson product $\langle i_p,i_t\rangle$ factors through $S^{2t+1}$, and $\langle i_p,i_p\rangle$ factors through $S^3$.
\end{lem}

\begin{proof}
We will deal with $\langle i_p,i_p\rangle$ first.
We know that 
\begin{equation*}
	SU(p+t-1)=\simeq B_2\times\ldots B_t\times S^{2t+1}\times\ldots S^{2p-1}
\end{equation*}
and so by \cite{MR2506127} we know that $\pi_{4p-2}(S^{2k-1})=0$ for $t+1\leq k\leq p$ and $\pi_{4p-2}(B_k)=0$ for $3\leq k\leq t$.
We also know that $\pi_{4p-2}(B_2)=\Zp$.
Therefore $\langle i_p,i_p\rangle$ must factor through $B_2$ and has order at most $p$.
Kishimoto \cite{MR2506127} then tells us that any map $X\rightarrow B_2$ of order $p$ lifts to $S^3$.
Hence $\langle i_p,i_p\rangle$ factors through $S^3$.

The case for $\langle i_p,i_t\rangle$ is slightly more involved.
Let $SU_{k}$ denote $SU(p+t-1)$ with the $k^{\textrm{th}}$ factor omitted.
For example
\begin{equation*}
	SU_{t+1}=B_2\times\ldots\times B_t\times S^{2t+3}\times\ldots\times S^{2p-1}.
\end{equation*}
and $SU(p+t-1)=SU_{t+1}\times S^{2t+1}$.
Let $f:S^{2(p+t-1)}\rightarrow A_t\wedge S^{2p-1}$ be the inclusion of the bottom cell and consider the composition
\begin{equation*}
	\xymatrix{ S^{2(p+t-1)}       \ar[d]^-{f} &          \\
    	         A_t\wedge S^{2p-1} \ar[r] & SU_{t+1} }
\end{equation*}
By \cite{MR2506127}, $\pi_{2(p+t-1)}(SU_{t+1})\cong 0$, meaning that we get an extension to the top cell of $A_t\wedge S^{2p-1}$:
\begin{equation*}
	\xymatrix{   S^{2(p+t-1)}       \ar[d]        &         \\
    	         A_t\wedge S^{2p-1} \ar[r] \ar[d] & SU_{t+1} \\
        	     S^{2(2p+t-2)}      \ar[ur]       &         }
\end{equation*}
By \cite{MR2506127}, $\pi_{2(2p+t-2)}(SU_{t+1})\cong 0$.
Thus
\begin{equation*}
	A_t\wedge S^{2p-1}\xrightarrow{\langle i_k,i_j\rangle} SU(p+t-1)\simeq SU_{t+1}\times S^{2t+1}
\end{equation*}
projects trivially onto $SU_{t+1}$. Therefore $\langle i_p,i_t\rangle$, must factor through $S^{2t+1}$.
\end{proof}

We deal now with length three Samelson products in $SU(p+t-1)$.
The length three Samelson products are somewhat easier to look at than the length two Samelson products.
We first note from Kishimoto's paper \cite{MR2506127} that $\langle i_j,\langle i_k,i_l\rangle\rangle$ is non-trivial only when $j+k+l$ is equal to either $2p+1$, $2p+2$, $2p+3$ or $3p$.
It is easy to see that no length three Samelson product can have order greater than $p^2$ if $t\neq p$ because
\begin{equation*}
    \langle i_j,\langle i_k,i_l\rangle \rangle \circ \underline{p}^m \simeq \langle i_j,\langle i_k,i_l\rangle\circ \underline{p}^m\rangle
\end{equation*}
and $\langle i_k,i_j\rangle$ has order $\leq p^2$ if $t\neq p$ by Propositions \ref{sama} and \ref{samb} and Lemma \ref{samb2}.

We prove the following lemma.
\begin{lem}\label{genlem2}
	Let $X$ be a CW-complex with cells only in dimensions \newline $2(m+k(p-1))-3$ where
	\begin{itemize}
  		\item $k\geq 1$,
  		\item $m\in\{3,4,5,6\}$ and
  		\item $2(m+k(p-1))-3\leq 12p-1$.
	\end{itemize}
	Then any map $X\rightarrow SU(p+t-1)$ factors through $S^{2m-1}$.
\end{lem}

\begin{proof}
As will be made clear from the proof, it is sufficient to assume that all the cells of $X$ are in different dimensions.
We proceed by induction on the dimension of $X$ and as the base case take $X$ to be $S^{2(m+p)-5}$.
Recall that $SU_m$ is $SU(p+t-1)$ with the $m^{\textrm{th}}$ factor omitted.
As all homotopy groups of the form $\pi_{2(m+p)-5}(SU_m)$ for $2(m+p)-5\leq 12p-1$ are zero \cite{MR2506127} we see that any map
\begin{equation*}
	S^{2(m+p)-5}\rightarrow SU(p+t-1)
\end{equation*}
must factor through the $m^{\textrm{th}}$ factor of $SU(p+t-1)$, which is $S^{2m-1}$ if $t+1\leq m\leq p$ or $B_m$ otherwise.
If it factors through $B_m$ consider the composition
\begin{equation*}
	S^{2(m+p)-5}\rightarrow B_m\xrightarrow{q} S^{2(m+p)-3}
\end{equation*}
where $q$ comes from the fibration
\begin{equation*}
	S^{2m-1}\rightarrow B_m\xrightarrow{q} S^{2m+2p-3}.
\end{equation*}
Then since $\pi_{2(m+p)-5}(S^{2(m+p)-3})=0$ this map lifts to $S^{2m-1}$.
This completes the base case.

For the inductive step let $X$ have dimension $2(m+k(p-1))-3$ where $k>1$. 
Let $X'$ be the $(2(m+k(p-1))-4)$-skeleton of $X$.
Then by the induction hypothesis the composition
\begin{equation*}
	X'\hookrightarrow X\rightarrow SU_m
\end{equation*}
is null homotopic.
Therefore we get an extension
\begin{equation*}
	\xymatrix{   X'  \ar[d]                         &      \\
				 X                 \ar[r] \ar[d]    & SU_m \\
				 S^{2(m+k(p-1))-3} \ar[ur]^-f       &      }
\end{equation*}
for some map $f$.
By \cite{MR2506127} $\pi_{2(m+k(p-1))-3}(SU_m)=0$.
Therefore $f$ is also null homotopic, implying that any map $X\rightarrow SU(p+t-1)$ must factor through the $m^{\textrm{th}}$ factor of $SU(p+t-1)$, which is $S^{2m-1}$ if $t+1\leq m\leq p$ or $B_m$ otherwise.
If the map factors through $B_m$ we proceed as in the base case and get a lift to $S^{2m-1}$.
Therefore we get the required result.
\end{proof}

\begin{cor}\label{sam3}
	Let $k+j+l=2p+1$, $2p+2$ or $2p+3$. 
	Then the Samelson $\langle i_k,\langle i_j,i_l\rangle\rangle$ factors through $S^{2(k+j+l-2p)-5}$.
	If $k+j+l=3p$, then Samelson product $\langle i_k,\langle i_j,i_l\rangle\rangle$ factors through $S^{3}$.
	\qed  
\end{cor}

We exclude the case where $t=p$.
This case represents another increase in technical difficulty and as yet cannot be resolved using this method.
The increase in difficulty lies in finding a factorization of $\langle i_p,i_p\rangle$ similar to those for the other Samelson products.
It is probably the case that when $t=p$ the Samelson product $\langle i_p,i_p\rangle:A_p^{(2)}\rightarrow SU(2p-1)$ has order $\leq p^3$ and factors through $A_2$.
The first of these assertions can be shown using the methods above.
The second assertion presents some difficulty.
If this factorization can be shown however, then it would be relatively simple to make slight alterations to the subsequent arguments to show that the $p^3$ power map on $SU(2p-1)$ is an H-map.

\section{Proof Of Theorem \ref{a}}

Recall that
\begin{equation*}
	A= A_2\vee\ldots\vee A_p.
\end{equation*}
Including the wedge $\bigvee\Sigma A_i$ into the product $\prod \Sigma A_i$ we obtain a homotopy fibration
\begin{equation*}
	Q\xrightarrow{\overline{f}} \Sigma A\rightarrow \prod_{k=2}^{p}\Sigma A_k
\end{equation*}
which defines the space $Q$ and the map $\overline{f}$.
Consider the homotopy fibration
\begin{equation*}
	\Omega Q\xrightarrow{\Omega\overline{f}} \Omega\Sigma A\rightarrow \prod_{k=2}^{p}\Omega\Sigma A_k.
\end{equation*}

\begin{thm}[Hilton-Milnor]\label{HM}
	There are homotopy equivalences
	\begin{equation*}
		\Omega\Sigma A\simeq \left(\prod_{k=2}^{p}\Omega\Sigma A_k\right) \times \Omega Q
	\end{equation*}
	and
	\begin{equation*}
		\Omega Q=\prod_{j\in\mathcal{J}} \Omega\Sigma \left( A_2^{(j_2)}\wedge A_3^{(j_3)}\wedge ... A_p^{(j_p)}\right)
	\end{equation*}
	where $\mathcal{J}$ runs over an additive basis of the free Lie algebra $L\langle u_2,\ldots,u_p\rangle$, but without the basis elements $u_2,\ldots,u_p$.
	Further, if for $2\leq r\leq p$ the map
	\begin{equation*}
		s_r:\Sigma A_r \rightarrow \bigvee \Sigma A
	\end{equation*}
	is the inclusion of the $r^{\textrm{th}}$-wedge summand, then the map $\Omega\overline{f}$ restricted to
	\begin{equation*}
		\Omega\Sigma\left(A_2^{(j_2)}\wedge\ldots\wedge A_p^{(j_p)} \right)
	\end{equation*}
	is the loops on the iterated Whitehead product of the maps $s_r$ corresponding to the index $j\in \mathcal{J}$.
	\qed
\end{thm}

For $2\leq k\leq p$, let $S_k$ be the $k^{\textrm{th}}$ factor of $SU(p+t-1)$.
Note that $S_k=B_k$ if $2\leq k\leq t$ and $S_k=S^{2k-1}$ if $t+1\leq k\leq p$.
Then a construction of Cohen and Neisendorfer \cite{MR764588} allows us to produce a fibration over each factor of $SU(p+t-1)$.
Let $[f,g]$ be the Whitehead product of maps $f$ and $g$.

\begin{thm}[Cohen-Neisendorfer]\label{CN}
	There exists a fibration
	\begin{equation*}
		\Omega R_k\xrightarrow{\Omega f_k} \Omega\Sigma A_k \xrightarrow{\iota_k} S_k
	\end{equation*}
	such that
	\begin{itemize}
		\item $R_k$ is a retract of $A_k^{(2)}\vee A_k^{(3)}$ if $S_k=B_k$ or
		\item $R_k=S^{4k-1}$ if $S_k=S^{2k-1}$.
	\end{itemize}
	The map $f_k$ factors through 
	\begin{itemize}
		\item $\Sigma A_k^{(2)}\xrightarrow{[ i_k,i_k]} \Sigma A_k$ if $S_k=B_k$ or
		\item $\Sigma A_k^{(2)}\vee A_k^{(3)}\xrightarrow{[ i_k,i_k]\vee[ i_k,[ i_k,i_k]]} \Sigma A_k$ if $S_k=S^{2k-1}$.
	\end{itemize}	
	Furthermore, $\Omega f_k$ has a left homotopy inverse and so there is a homotopy decomposition
	\begin{equation*}
		\Omega\Sigma A_k \simeq S_k\times \Omega R_k.
	\end{equation*}
\end{thm}

Combining Theorems \ref{HM} and \ref{CN}, and noting that $SU(p+t-1)\simeq \prod_{k=2}^{p}S_k$, we get the following.

\begin{lem}\label{same}
	There exists a homotopy equivalence
	\begin{equation*}
		\Omega\Sigma A \simeq SU(p+t-1) \times \prod_{k=2}^{p} \Omega R_k \times \Omega Q.
	\end{equation*}
	\qed
\end{lem}

A consequence of Cohen and Neisendorfer's work \cite{MR764588} is that the map $\Omega i$ has a right homotopy inverse.
Therefore we know that there is a decomposition.

\begin{lem}\label{decomp}
	$\Omega\Sigma A\simeq SU(p+t-1)\times \Omega F$.
	\qed
\end{lem}

We then get the following corollary.

\begin{cor}
	$\Omega F$ and $\prod_{k=2}^{p} \Omega R_k \times \Omega Q$ have the same homotopy type.
\end{cor}

Theorems \ref{HM} and \ref{CN} also state that the factors $\Omega R_k$, $2\leq k\leq p$ and $\Omega Q$ map to $\Omega\Sigma A$ through looped Whitehead products. 
In what follows we wish to show to what extent these factors lift through the map $\Omega\nu$ in the fibration
\begin{equation*}
	\Omega F\xrightarrow{\Omega\nu} \Omega\Sigma A\xrightarrow{\Omega i} SU(p+t-1)
\end{equation*}
and amend them to produce lifts when obstructions exist.
This first requires some notation and preliminary results.

A space $H$ is an \emph{H-group} if it is an H-space whose multiplication also has a homotopy inverse.
Let $H$ be an H-group, and let $c:H\times H\rightarrow H$ be the commutator.
Pointwise, $c$ is defined by $c(a,b)=aba^{-1}b^{-1}$.
We can then iterate this and get for $k\geq 1$ the ``$k$-fold commutator map'', $c_k:H^{k+1}\rightarrow H$ defined by
\begin{equation*}
c_k: c\circ(1\times c)\circ...\circ(1\times1\times...\times1\times c).
\end{equation*}

\begin{dfn}
    An H-group $H$ has homotopy nilpotence class $k$, denoted \mbox{nil$(H)=k$}, if and only if $c_k$ is null homotopic and $c_{k-1}$ is not.
\end{dfn}
If an H-group $H$ has homotopy nilpotence class $k$, then any length \mbox{$k+1$} Samelson products in $H$ are trivial.
We can now state a theorem of Kishimoto \cite{MR2506127}.
\begin{thm}[Kishimoto]\label{kish}
	Let $p$ be a prime greater than 5.
	Then
	\begin{enumerate}
	\item nil$(SU(n))=3$ if $p=n+1$ or $\frac{n}{2}< p\leq\frac{2n+1}{3}$ and
	\item nil$(SU(n))=2$ if $\frac{2n+1}{3}<p\leq n-2$.
	\end{enumerate}
	\qed
\end{thm}

By Theorem \ref{kish} we also know the following.

\begin{lem}\label{nulllength4}
	Let $\omega_k$ be a length $k$ Whitehead product. Then
	\begin{enumerate}
		\item if $p=n+1$ or $\frac{n}{2}< p\leq\frac{2n+1}{3}$, $(\Omega 	i)\circ(\Omega\omega_k)$ is trivial for all $k\geq 4$,
		\item if $\frac{2n+1}{3}<p\leq n-2$, $(\Omega i)\circ(\Omega\omega_k)$ is trivial for all $k\geq 4$.
	\end{enumerate}
\end{lem}

Now consider lifts of certain Samelson products through the map $\Omega\nu$ in the fibration 
\begin{equation*}
\Omega F\xrightarrow{\Omega \nu} \Omega \Sigma A \xrightarrow{\Omega i} SU(p+t-1).
\end{equation*}

\begin{lem}\label{newlem}
	Let $j\geq 4$ and $\theta_k: A_k\hookrightarrow A \xrightarrow{E} \Omega\Sigma A$ where $E$ is the suspension map. Then the iterated Samelson product $\langle\theta_{k_1},\langle\theta_{k_2},\ldots\langle\theta_{k_{j-1}},\theta_{k_j}\rangle\ldots\rangle$ lifts through the map $\Omega\nu$.
\end{lem}

\begin{proof}
We will show that $\Omega i\circ \langle\theta_{k_1},\langle\theta_{k_2},\ldots,\langle\theta_{k_{j-1}},\theta_{k_j}\rangle\ldots\rangle\rangle$ is trivial.
Since $\Omega i$ is an H-map, $\Omega i\circ \langle\theta_{k_1},\langle\theta_{k_2},\ldots,\langle\theta_{k_{j-1}},\theta_{k_j}\rangle\ldots\rangle\rangle$ is homotopic to $\langle\Omega i\circ\theta_{k_1},\langle\Omega i\circ\theta_{k_2},\ldots,\langle\Omega i\circ\theta_{k_{j-1}},\Omega i\circ\theta_{k_j}\rangle\ldots\rangle\rangle$.
As $\Omega i\circ \theta_k$ is the inclusion of $A_k$ into $SU(p+t-1)$, it is precisely the map $i_k$ of the previous section.
Therefore, by Theorem \ref{kish} all Samelson products of length $4$ or more compose trivially into $SU(p+t-1)$, and so lift to $\Omega F$.
\end{proof}

We must now deal with the Samelson products of length $2$ and $3$.

First we deal with the length $2$ case.
Since $\Omega i$ is an H-map, $\Omega i\circ \langle\theta_k,\theta_j\rangle$ is homotopic to $\langle\Omega i\circ\theta_k,\Omega i\circ\theta_j\rangle$.
Therefore if $\langle i_k,i_j\rangle$ is trivial then $\langle\theta_k,\theta_j\rangle$ lifts to $\Omega F$.
We know from Lemma \ref{sama} that if $k+j\leq p+1$ then $\langle i_k,i_j\rangle$ is trivial.
For $p+2\leq k+j\leq 2p$ however, $\langle i_k,i_j\rangle$ may not be trivial.
In the case that it is not trivial, by Lemmas \ref{sama} and \ref{samfacb} it is homotopic to a composite 
\begin{equation*}
	\alpha_{k,j}:A_k\wedge A_j\xrightarrow{f} S^{2(k+j-p+1)-1}\rightarrow SU(p+t-1)	
\end{equation*}
unless $t=p$. 

Let $a_{k,j}$ be the composite
\begin{equation*}
	a_{k,j}:A_k\wedge A_j\xrightarrow{f} S^{2(k+j-p+1)-1}\hookrightarrow A \xrightarrow{E} \Omega\Sigma A,
\end{equation*}
then $\alpha_{k,j}\simeq (\Omega i)\circ a_{k,j}$.
Therefore the difference $\overline{\eta}_{k,j}=\langle \theta_k,\theta_j\rangle-a_{k,j}$ composes trivially with $\Omega i$ and so lifts to $\Omega F$,
\begin{equation*}
	\xymatrix{ & A_k\wedge A_j \ar[d]^-{\overline{\eta}_{k,j}} \ar[dl]_-{\mu_{k,j}}  & \\
    	        \Omega F \ar[r]^-{\Omega\nu} & \Omega \Sigma A \ar[r]^-{\Omega i} & SU(p+t-1) }.
\end{equation*}
for some map $\mu_{k,j}$.

Using Lemma \ref{sam3} we can define similar maps, 
\begin{equation*}
	a_{k,j,l}:A_k\wedge A_j\wedge A_l\rightarrow S^{2m-1}\hookrightarrow A \xrightarrow{E} \Omega\Sigma A
\end{equation*}
for non-trivial length $3$ Samelson products with the property that 
\begin{equation*}
	\langle i_j,\langle i_k,i_l\rangle\rangle\simeq (\Omega i)\circ a_{k,j,l}.
\end{equation*}
This gives us differences $\overline{\eta}_{j,k,l}=\langle i_j,\langle i_k,i_l\rangle\rangle - a_{k,j,l}$ which lift through $\Omega\nu$ to maps $A_k\wedge A_j\wedge A_l\xrightarrow{\mu_{k,j,l}}\Omega F$. 

\begin{lem}\label{trivimage}
	The image in homology of the maps $\overline{\eta}_{k,j}$ and $\overline{\eta}_{k,j,l}$ is equal to the image in homology of the maps $\langle i_k,i_j\rangle$ and $\langle i_k,\langle i_j,i_l\rangle\rangle$ respectively.
\end{lem}

\begin{proof}
We will show that $a_{k,j}$ and $a_{k,j,l}$ induce the zero map in homology.
Each $a_{k,j}$ and $a_{k,j,l}$ factors through some sphere $S^{2m-1}$ where $2\leq m\leq p$. So it suffices to show that the maps $A_k\wedge A_j\rightarrow S^{2m-1}$ and $A_k\wedge A_j\wedge A_l\rightarrow S^{2m-1}$ are zero in homology.
As each $A_k\wedge A_j$ and $A_k\wedge A_j\wedge A_l$ has cells in dimensions $\geq 2p+2$ or $\geq 4p$ respectively, this is clear for dimensional reasons.
\end{proof}

One result we will make repeated use of is a theorem of James \cite{MR0073181}.
We state it here to aid the reader.
\begin{thm}[James]\label{James}
	Let $f:X\rightarrow Y$ be some map where $Y$ is a homotopy associative $H$-space. 
	Then $f$ extends to an H-map $g:\Omega\Sigma X \rightarrow Y$ where $g$ is the unique H-map such that $g\circ E\simeq f$ where $E$ is the suspension functor.
	\qed
\end{thm}



Using Theorem \ref{James} we get H-maps 
\begin{align*}
	\eta_{k,j}   &: \Omega\Sigma (A_k\wedge A_j) \rightarrow \Omega \Sigma A, \\
	\eta_{k,j,l} &: \Omega\Sigma (A_k\wedge A_j\wedge A_l) \rightarrow \Omega \Sigma A,
\end{align*}
extending $\overline{\eta}_{k,j}$ and $\overline{\eta}_{k,j,l}$ respectively.
These maps are then the loops on Whitehead products minus some correcting factor.

\begin{lem}\label{1}
	The maps $\eta_{k,j}$, $\eta_{k,j,l}$ lift through $\Omega F\xrightarrow{\Omega\nu} \Omega\Sigma A$.
\end{lem}

\begin{proof}
It is equivalent to show that the maps $(\Omega i)\circ\eta_{k,j}$ and $(\Omega i)\circ \eta_{k,j,l}$ are null homotopic. Since $\Omega i$, $\eta_{k,j}$ and $\eta_{k,j,l}$ are H-maps, then by Theorem \ref{James} the homotopy  class of $(\Omega i)\circ \eta_{k,j}$ and $(\Omega i)\circ \eta_{k,j,l}$ are determined by the restrictions
\begin{equation*}
	A_k\wedge A_j \xrightarrow{E} \Omega\Sigma(A_k\wedge A_j)\xrightarrow{\eta_{k,j}}\Omega\Sigma A \xrightarrow{\Omega i} SU(p+t-1)
\end{equation*}
and
\begin{equation*}
	A_k\wedge A_j\wedge A_l \xrightarrow{E} \Omega\Sigma(A_k\wedge A_j\wedge A_l)\xrightarrow{\eta_{k,j,l}}\Omega\Sigma A \xrightarrow{\Omega i} SU(p+t-1)
\end{equation*}
respectively.
These restrictions are $\overline{\eta}_{k,j}$ and $\overline{\eta}_{k,j,l}$.
By construction we know that $(\Omega i)\circ\overline{\eta}_{k,j}$ $(\Omega i)\circ\overline{\eta}_{k,j,l}$ are trivial.
Therefore $(\Omega i)\circ\eta_{k,j}$ and $(\Omega i)\circ\eta_{k,j,l}$ are trivial.
\end{proof}

We can also state the image in homology of $\eta_{k,j}$ and $\eta_{k,j,l}$.

\begin{lem}\label{2}
	The maps induced in homology by $\eta_{k,j}$ and $\eta_{k,j,l}$ are equal to those induced by $\Omega[i_k,i_j]$ and $\Omega[i_k,[i_j,i_l]]$ respectively.
\end{lem}

\begin{proof}
First note that the images of $\Omega[i_k,i_j]$ and $\Omega[i_k,[i_j,i_l]]$ in homology are the multiplicative extensions of $(\Omega\langle i_k,i_j\rangle)_{\ast}$ and $(\Omega\langle i_k,\langle i_j,i_l\rangle\rangle)_{\ast}$.
By Theorem \ref{James} we know that $(\eta_{k,j})_{\ast}$ and $(\eta_{k,j,l})_{\ast}$ are multiplicative extensions of $(\overline{\eta}_{k,j})_{\ast}$ and $(\overline{\eta}_{k,j,l})_{\ast}$.
The proof is completed by applying Lemma \ref{trivimage}.
\end{proof}

We can now define a map very similar to that from the Hilton-Milnor Theorem.
Let $k_i\in \{2,\ldots,p\}$.
Define a map $\lambda:\Omega Q\rightarrow \Omega\Sigma A$ by its restriction to the factors $\Omega\Sigma(A_{k_1}\wedge\ldots \wedge A_{k_m})$ of $\Omega Q$ as
\begin{itemize}
	\item $\Omega[i_{k_1},[i_{k_2},\ldots,[i_{k_{m-1}},i_{k_m}]]]$ if $m\geq 4$,
	\item $\Omega[i_{k_1},[i_{k_2},\ldots,[i_{k_{m-1}},i_{k_m}]]]$ if $m=3$ and $2\leq t< \frac{p+1}{2}$,
	\item $\eta_{k_1,k_2,k_3}$ if $m=3$ and $\frac{p+1}{2}\leq t\leq p$ and
	\item $\eta_{k_1,k_2}$ if $m=2$. 
\end{itemize}

In particular, $\lambda=\Omega\overline{f}$ if $m\geq 4$ or if $m=3$ and $2\leq t\leq \frac{p+1}{2}$.
In the two other cases by Lemma \ref{2} we have $(\eta_{k,j,l})_{\ast}=(\Omega[i_k,[i_j,i_l]])_{\ast}$ and $(\eta_{k,j})_{\ast}=(\Omega[i_k,i_j])_{\ast}$.
So collectively we obtain the following.

\begin{lem}
	$\lambda_{\ast}=(\Omega\overline{f})_{\ast}$.
\qed
\end{lem}

Putting Theorems \ref{1} \ref{nulllength4} together we obtain a homotopy commutative triangle
\begin{equation*}
	\xymatrix{ \empty                       & \Omega Q       \ar[d]^-{\lambda} \ar[dl]_-{g'} \\ 
               \Omega F \ar[r]^-{\Omega\nu} & \Omega\Sigma A                                 }
\end{equation*}
for some lift $g'$.

Next, Theorem \ref{CN} lets us deal with the maps $\Omega f_k:\Omega R_k\rightarrow \Omega\Sigma A_k$ with the information we have about the Whitehead products.
The map $f_k$ is either a sum of maps which factor through a length two Whitehead product or through a wedge sum of a length two and length three Whitehead products. 
If $2k\geq p+2$ then we replace a length two Whitehead product by the adjoint of the difference $\langle i_k,i_k\rangle -a_{k,k}$.
If $3k= 2p+1,2p+2,2p+3$ or $3p$ then we replace the length three Whitehead product factor of $f_k$ by the adjoint of the difference $\langle i_k,\langle i_k,i_k\rangle\rangle -a_{k,k,k}$.
Amending these as before gives us a map $\overline{f}_k:R_k\rightarrow \Sigma A_k\hookrightarrow \Sigma A$ such that the composition
\begin{equation*}
	\Omega R_k\xrightarrow{\Omega\overline{\eta}_k} \Omega\Sigma A_k \hookrightarrow \Omega\Sigma A \xrightarrow{\Omega i} SU(p+t-1)
\end{equation*}
is trivial.

Following the proofs of Lemmas \ref{1} and \ref{2} we get the following result.

\begin{lem}\label{3}
	The map $\overline{f}_k$ lifts through $\Omega\nu$ and has image in homology equal to the image in homology of the composition 
	\begin{equation*}
		\Omega R_k\xrightarrow{\Omega f}\Omega\Sigma A_k\hookrightarrow \Omega\Sigma A.
	\end{equation*}
\end{lem}

\begin{proof}
The proof that $\overline{f}_k$ lifts through $\Omega\nu$ follows from the fact that $(\Omega i)\circ \overline{\eta}_{k}$ is trivial. The equality of the images in homology follows directly from Lemma \ref{trivimage}.
\end{proof}



Collecting the maps $\overline{f}_k$ together we get a map
\begin{equation*}
	\lambda':\prod_{k=2}^{p}\Omega R_k \xrightarrow{\prod\Omega\overline{f}_k} \prod_{k=2}^{p}\Omega\Sigma A \xrightarrow{\mu^{k-1}} \Omega\Sigma A
\end{equation*}
Where $\mu$ is the multiplication on $\Omega\Sigma A\times\Omega\Sigma A\rightarrow \Omega\Sigma A$.
We can then define a map $\rho:\Omega Q\times\prod_{k=2}^{p}\Omega R_k\rightarrow \Omega\Sigma A$ by the composite
\begin{equation*}
	\rho:\Omega Q\times\prod_{k=2}^{p}\Omega R_k\xrightarrow{\lambda\times\lambda'} \Omega\Sigma A\times\Omega\Sigma A\xrightarrow{\mu}\Omega\Sigma A. 
\end{equation*}

Putting Lemmas \ref{1} and \ref{3} together we obtain a commutative diagram
\begin{equation}\label{diag}
	\xymatrix{                   & \Omega Q\times\prod_{k=2}^{p}\Omega R_k \ar[d]^-{\rho} \ar[dl]_-{g} \\
						    \Omega F \ar[r]^-{\Omega\nu} & \Omega\Sigma A                                                      }
\end{equation}
for some lift $g$.
Also notice that by Lemmas \ref{2} and \ref{3} we have that $\rho_{\ast}=(\Omega\nu)_{\ast}$.

\begin{proof}[Proof of Theorem \ref{a}]
By Lemmas \ref{decomp} and \ref{same} we know that $\Omega F$ and $\Omega Q\times\prod_{k=2}^{p}\Omega R_k$ have the same homotopy type.
Therefore they both have the same Euler-Poincar\'{e} series.
By construction, $\rho_{\ast}$ is a monomorphism, hence so is $(\Omega\nu)_{\ast}$.
Therefore $g_{\ast}$ is a monomorphism between two $\Zp$ vector spaces with the same Euler Poincar\'{e} series.
Hence $g_{\ast}$ is an isomorphism, and so Whitehead's Theorem tells us that $g$ is a homotopy equivalence. 
\end{proof}

\section{Proof Of Theorem \ref{b}}

Before continuing we will need a small amount of setting up.
Let $Y$ be a homotopy associative H-space and suppose that there is a space $X$ and map $\overline{f}:X\rightarrow Y$ such that $H_{\ast}(Y)\cong \Lambda (\tilde{H}_{\ast}(X))$, with $\overline{f}_{\ast}$ inducing the inclusion of the generating set. 
By Theorem \ref{James} this extends to a map $f:\Omega\Sigma X\rightarrow Y$.
The map $f$ is an H-map and is the unique map such that $f\circ E\simeq \overline{i}$.
There is then a homotopy fibration
\begin{equation*}
	K\xrightarrow{h} \Omega\Sigma X \xrightarrow{f} Y 
\end{equation*}
defining the space $K$ and map $h$.

With this in mind we state a result of Theriault \cite{T}.
\begin{lem}\label{dia}
	Suppose that $f$ has a right homotopy inverse.
	Let $Z$ be a homotopy associative H-space and $\overline{e}:X\rightarrow Z$ be any map. 
	Theorem \ref{James} tells us that this extends uniquely to an H-map $e:\Omega\Sigma X\rightarrow Z$ such that $E\circ h\simeq \overline{e}$.
	If the composite $K\xrightarrow{h} \Omega\Sigma X\xrightarrow{e} Z$ is null homotopic then there exists a homotopy commutative diagram
	\begin{equation*}
		\xymatrix{ \Omega\Sigma X \ar@2{-}[d] \ar[r]^-{f} & Y \ar[d]^-{g} \\
    	           \Omega\Sigma X \ar[r]^-{e}            & Z            }
	\end{equation*}
	for some map $g$ which can be chosen to be an H-map.
	\qed
\end{lem}

We will use Lemma \ref{dia} to show the following lemma.
\begin{lem}\label{z}
	If $2\leq t <p$ then there exists a homotopy commutative diagram
	\begin{equation*}
		\xymatrix{ \Omega\Sigma A \ar[d]_-{\Omega\degppp} \ar[r]^-{\Omega i} & SU(p+t-1) 	\ar[d]^-{g} \\
        	       \Omega\Sigma A \ar[r]^-{\Omega i}                       & SU(p+t-1)            }
    \end{equation*}
	where $g$ can be chosen to be an H-map.
\end{lem}

\begin{proof}
By Lemma \ref{dia} it is enough to show that the composite
\begin{equation*}
	\Omega F \xrightarrow{\Omega\nu} \Omega\Sigma A \xrightarrow{\Omega\degppp} \Omega\Sigma A \xrightarrow{\Omega i} SU(p+t-1)
\end{equation*}
is null homotopic.
We prove the case when $\nil(SU(p+t-1))=2$, the case when $\nil(SU(p+t-1))=3$ being similar.
By Theorem \ref{a}, $\Omega F$ decomposes as
\begin{equation*}
	\prod_{j\in\mathcal{J}}\Omega\Sigma\left( A_2^{(j_2)}\wedge A_3^{(j_3)} \wedge\ldots\wedge A_t^{(j_t)}\wedge S^{(j_{t+1}(2t+1))}\wedge\ldots\wedge S^{(j_{p}(2p-1))} \right)\times\prod_{k=2}^{p} \Omega R_k.
\end{equation*}
Since $\Omega\nu$ is an H-map it is determined by the product of the restrictions to each factor of $\Omega F$.
First consider any factor indexed by $\mathcal{J}$ involving the smash of three or more spaces.
By Theorem \ref{HM} the restriction of $\Omega\nu$ to such a factor is a looped Whitehead product of length three or more.
The composite
\begin{equation*}
	\Omega\Sigma(A_{k_1}\wedge \ldots\wedge A_{k_t}) \xrightarrow{\Omega[\theta_{k_1},[\theta_{k_2},\ldots,[\theta_{k_t}]\ldots]]}\Omega\Sigma A \xrightarrow{\Omega\degppp} \Omega\Sigma A\xrightarrow{\Omega i} SU(p+t-1)
\end{equation*}
is of loop maps and so by Theorem \ref{James} is determined by its restriction to the wedge product $A_{k_1}\wedge A_{k_2}\wedge\ldots\wedge A_{k_t}$.
This restriction is the composite 
\begin{equation*}
	(\Omega i)\circ (\Omega\degppp)\circ \langle\theta_{k_1},\langle\theta_{k_2},\ldots\langle \theta_{k_{t-1}},\theta_{k_t}\rangle\ldots\rangle\rangle.
\end{equation*}
Then as $\Omega\degppp$ and $\Omega i$ are both H-maps this is homotopic to the Samelson product
\begin{equation*}
	\langle(\Omega i)\circ(\Omega\degppp)\circ\theta_{k_1},\langle(\Omega i)\circ(\Omega\degppp)\circ\theta_{k_2},\ldots,\langle(\Omega i)\circ(\Omega\degppp)\circ\theta_{k_{t-1}},(\Omega i)\circ(\Omega\degppp)\circ\theta_{k_t}\rangle\ldots\rangle\rangle.
\end{equation*}
As this is a Samelson product of length three or more it is null homotopic since $\nil(SU(p+t-1))=2$.

Next let $\Omega\Sigma(A_k\wedge A_j)$ be a factor in $\Omega Q$ involving two smash factors..
Then by Theorem \ref{a}, $\Omega\nu$ restricted to this factor is either a looped Whitehead product or $\eta_{k,j}$.
Therefore $\Omega\nu$ restricted to this factor is determined by the restriction to $(A_k\wedge A_j)$, the Samelson product $\langle \theta_k,\theta_j\rangle$.
Consider $(\Omega\degppp)\circ\langle \theta_k,\theta_j\rangle$.
The naturality of the suspension implies there is a homotopy commutative diagram
\begin{equation}\label{diag2}
	\xymatrix{ A_k \ar[r]^-{\theta_k} \ar[d]^-{\degppp}& \Omega\Sigma A \ar[d]^-{\Omega\Sigma\degppp} \\
    	       A_k \ar[r]^-{\theta_k}                 & \Omega\Sigma A.                        }
\end{equation}
This means we get the following string of homotopies, 
\begin{equation*}
	(\Omega\Sigma\degppp)\circ\langle\theta_k,\theta_j\rangle\simeq\langle(\Omega\Sigma\degppp)\circ\theta_k,(\Omega\Sigma\degppp)\circ\theta_j\rangle\simeq\langle\theta_k\circ\degppp,\theta_j\circ\degppp\rangle\simeq \langle\theta_k,\theta_j\rangle\circ\underline{p}^6.
\end{equation*}
As $\langle \theta_k,\theta_j\rangle$ has order at most $p^3$ by Propositions \ref{sama} and \ref{samb} and Lemma \ref{samb2}, we see that this string of homotopies means that $(\Omega\degppp)\circ\langle\theta_k,\theta_j\rangle$ is trivial. 
Note that as $\Omega\nu$ restricted to $\Omega R_k$ factors through a looped Whitehead product, we may deal with the factors $\Omega R_k$ in the same manner.

Now consider one of the amended looped Whitehead products.
Again we look at the restriction to $A_k\wedge A_j$ which is $\overline{\eta}_{k,j}=\langle\theta_k,\theta_j\rangle-a_{k,j}$.
Since $\Omega\Sigma\degppp$ is a loop map it distributes on the left.
So we obtain
\begin{equation*}
	(\Omega\degppp)\circ\left(\langle\theta_k,\theta_j\rangle-a_{k,j}\right)\simeq(\Omega\degppp)\circ\langle\theta_k,\theta_j\rangle-(\Omega\Sigma\degppp)\circ a_{k,j}\simeq\langle\theta_k,\theta_j\rangle\circ\underline{p}^6 - (\Omega\degppp)\circ a_{k,j}.
\end{equation*}
By Lemmas \ref{sama} and \ref{samfacb} the map $a_{k,j}$ compresses to $S^{2(k+j-p+t)-1}\subset \Omega\Sigma A$.
This means there is a homotopy commutative diagram
\begin{equation*}
	\xymatrix{ A_k\wedge A_j \ar[r]^-{a'_{k,j}}& S^{2(k+j-p+t)-1} \ar[d]^-{\degppp} \ar[r]^-{E'} &\Omega\Sigma A \ar[d]^-{\Omega\degppp} \\
             \empty                          & S^{2(k+j-p+t)-1} \ar[r]^-{E'}                  &\Omega\Sigma A }
\end{equation*}
where $E'$ is the composite $S^{2(k+j-p+t)-1}\hookrightarrow A\rightarrow \Omega\Sigma A$ and the top row is homotopic to $a_{k,j}$.

Consider the string of homotopies
\begin{equation*}
	(\Omega\degppp)\circ a_{k,j}\simeq (\Omega\degppp)\circ E'\circ a'_{k,j}\simeq E'\circ\degppp\circ a'_{k,j}\simeq p^3\circ(E'\circ a'_{k,j})\simeq p^3\circ a_{k,j}\simeq a_{k,j}\circ\degppp.
\end{equation*}
By Diagram \ref{diag2} we get the homotopy $(\Omega\degppp)\circ E'\circ a'_{k,j}\simeq E'\circ\degppp\circ a'_{k,j}$.
Since $2(k+j-p+t)-1\geq 3$, $S^{2(k+j-p+t)-1}$ is a suspension and so the two products in $[S^{2(k+j-p+t)-1},\Omega\Sigma A]$ agree, implying that we get the second homotopy $E'\circ\degppp\circ a'_{k,j}\simeq p^3\circ(E'\circ a'_{k,j})=p^3\circ a_{k,j}$.
Now recall by Proposition \ref{sama} that the amended Whitehead products come from factors $A_k\wedge A_j$ where $k+j\geq p+2$. 
Therefore at least one of $A_k$ or $A_j$ is a suspension, hence so is $A_k\wedge A_j$. 
Thus the two products in $[A_k\wedge A_j,\Omega\Sigma A]$ agree. 
We therefore finally have a homotopy $p^3\circ a_{k,j}\simeq a_{k,j}\circ\degppp$.

Therefore
\begin{equation*}
	(\Omega\degppp)\circ\left(\langle\theta_k,\theta_j\rangle-a_{k,j}\right)\simeq \langle\theta_k,\theta_j\rangle\circ\underline{p}^6 - a_{k,j}\circ\degppp.
\end{equation*}
Since the composition of $\langle\theta_k,\theta_j\rangle$ and $a_{k,j}$ with $\Omega i$ has order at most $p^3$ by Lemmas \ref{sama} and \ref{samb} we see that $(\Omega i)\circ(\Omega\degppp)\circ(\langle\theta_k,\theta_j\rangle-a_{k,j})$ is null homotopic also.

We have not looked at the case where the factor is $\Omega R_k$ yet. However, as Theorem \ref{CN} tells us that these factor through the loops on Whitehead products, they are dealt with by the arguments above.
This concludes the proof.
\end{proof}

To prove Theorem \ref{b} it remains to show that the map $g$ in Lemma \ref{z} is homotopic to the $p^3$ power map.
To show this we will consider the difference $\Omega\degppp- p^3$ on $\Omega\Sigma A$.
Barratt \cite{MR0120647} examined this difference in some detail.

Consider the composite 
\begin{equation*}
	\Sigma X \xrightarrow{\sigma} \bigvee_{j=1}^{k}\Sigma X \xrightarrow{\nabla} \Sigma X,
\end{equation*}
where $\sigma$ is the $(k-1)$-fold diagonal map and $\nabla$ is the $(k-1)$-fold folding map.
After looping and applying the Hilton-Milnor Theorem to the space $\Omega\left(\bigvee_{j=1}^{k} \left(\Sigma X\right)\right)$ for some path-connected space $X$ and then folding we obtain the formula
\begin{equation*}
	\Omega \underline{k}\simeq k + \sum_{j=2}^{\infty}n_j(\Omega\omega_j)\circ H_{j}
\end{equation*}  
where $\omega_j$ is a length $j$ Whitehead product of the identity map on $\Sigma X$ with itself, $n_j$ is some integer, and $H_{j}$ is the Hilton-Hopf invariant.
This is known as the Distributive Law and is due to Barratt \cite{MR0120647}.
Barratt's construction is in terms of homotopy groups but is easily rephrased in terms of spaces.
If we take $k=p$ for a prime $p$, Barratt showed that the integers $n_j$ are divisible by $p$ if $j<p$.
In our case we require $\eta_2$ and $\eta_3$ to be divisible by $p^3$.

If we examine the detailed construction of the Hilton-Milnor theorem, for example in Whitehead's book \cite{MR516508}, it is seen that the number of length $j$ Whitehead products in the Hilton-Milnor equivalence is divisible by the number of basic products of length $j$ on $k$ elements.
The number of these basic products can be found in Whitehead's book \cite{MR516508}.
\begin{lem}
	The number of basic products of length $j$ on $k$ elements is given by
	\begin{equation*}
		\frac{1}{j}\sum_{d|j}\mu(d)k^{j/d} 
	\end{equation*}
	where $\mu$ is the M\"{o}bius inversion function, which is defined as $\mu(1)=1$, $\mu(d)=(-1)^l$ if $d=p_1p_2...p_l$ for unique primes $p_k$, and $\mu(d)=0$ otherwise.
	\qed
\end{lem}
Notice that if we take $k=p^m$ for some integer $m$, then this lemma implies that the number of length $j$ basic products is divisible by $p^m$ if $j<p$.
So by taking each of the $k$ spaces in the Hilton-Milnor theorem to be the same space $X$ then we see that the number of occurrences of the Whitehead product $\omega_j$ is divisible by $p^m$ if $j<p$.
Hence we can state the following version of the Distributive Law.
\begin{lem}\label{dist}
	Let $\omega_j$ be a length $j$ Whitehead product.
	Then 
	\begin{equation*}
		\Omega\degpm \simeq p^m + \sum_{j=2}^{\infty}n_j\left((\Omega\omega_j)\circ H_{j}\right)
	\end{equation*}
	were $n_j$ is divisible by $p^m$ if $j< p$.
	\qed
\end{lem}

\begin{prop}
	The composite $\Omega\Sigma A\xrightarrow{\Omega\degppp-p^3}\Omega\Sigma A\xrightarrow{\Omega i} SU(p+t-1)$ is null homotopic if $2\leq t<p$.
\end{prop}

\begin{proof}
For $2\leq j\leq p-1$, let $\overline{n}_j=\frac{n_j}{p^3}$. 
By Lemma \ref{dist}, $\overline{\eta}_{j}\in \Z$.
Also we have
\begin{align*}
	(\Omega i)\circ(\Omega\degpp-p^3)&\simeq (\Omega i)\circ\left(\sum_{j=2}^{p-1}\left(p^3\overline{n}_j\circ(\Omega\omega_j)\circ H_{j}\right) + \sum_{j=p}^{\infty}n_j\left((\Omega\omega_j)\circ H_{j}\right)\right)\\
								  &\simeq \sum_{j=1}^{p-1}\left((\Omega i)\circ p^2\overline{n}_j\circ(\Omega\omega_j)\circ H_{j}\right) + \sum_{j=p}^{\infty}\left((\Omega i)\circ n_j(\Omega\omega_j)\circ H_{j}\right)
\end{align*}
For $j>3$ the composite $(\Omega i)\circ(\Omega\omega_j)$ is determined by its restriction $(\Omega i)\circ(\Omega\omega_j)\circ E$ by Theorem \ref{James}.
This restriction is a length $j$ Samelson product in $SU(p+t-1)$ and is therefore trivial by Propositions \ref{sama} and \ref{samb} and Lemma \ref{samb2}.

When $j\leq 3$, then $j<p$ as $p>5$.
The composite $(\Omega i)\circ p^3\circ(\Omega\omega_j)$ is again determined by its restriction $(\Omega i)\circ p^3\circ(\Omega\omega_j)\circ E$.
As the $p^{\textrm{th}}$ power map commutes with H-maps, we get that $(\Omega i)\circ p^3\circ(\Omega\omega_j)\circ E\simeq (\Omega i)\circ(\Omega\omega_j)\circ p^3\circ E$.
This however is homotopic to
\begin{equation*}
	(\Omega i)\circ(\Omega\omega_j)\circ p^3\circ E\simeq (\Omega i)\circ(\Omega\omega_j)\circ E \circ \degpp
\end{equation*}
Since the composite $(\Omega i)\circ(\Omega\omega_j)\circ E$ is a Samelson product in $SU(p+t-1)$ of length less than four, we known that it is of order at most $p^3$ by Propositions \ref{sama} and \ref{samb} and Lemma \ref{samb2}.
Therefore $(\Omega i)\circ p^3\circ(\Omega\omega_j)\circ E$ is trivial.
Hence $(\Omega i)\circ(\Omega\degpp-p^3)$ is trivial as required.
\end{proof}

\begin{cor}\label{last}
	The composite $\Omega\Sigma A\xrightarrow{p^3} \Omega\Sigma A\xrightarrow{\Omega i} SU(p+t-1)$ is homotopic to the composite $\Omega\Sigma A\xrightarrow{\Omega\degppp} \Omega\Sigma A\xrightarrow{\Omega i} SU(p+t-1)$.
\qed
\end{cor}

With all of this in mind we can finally prove our main theorem.

\begin{proof}[Proof of Theorem 2.2]
Recall that by Lemma \ref{decomp} there is a right homotopy inverse to the map $\Omega i$ which we will call $r$.
Therefore by Lemma \ref{dia} we get a homotopy commutative diagram
\[\xymatrix{ SU(p+t-1) \ar[r]^-{r}& \Omega\Sigma A \ar[d]^-{\Omega\degppp} \ar[r]^-{\Omega i} & SU(p+t-1) \ar[d]^-{g} \\
             \empty               & \Omega\Sigma A \ar[r]^-{\Omega i}                       & SU(p+t-1)            }\]
where $g$ is an H-map.
Chasing around the diagram we see that 
\begin{equation*}
(\Omega i)\circ(\Omega\degppp)\circ r \simeq g\circ(\Omega i)\circ r\simeq g.
\end{equation*}
From Corollary \ref{last} we see that $(\Omega i)\circ(\Omega\degppp)\circ r\simeq (\Omega i)\circ p^3 \circ r$.
Since power maps commute with H-maps we get the following string of homotopies.
\begin{equation*}
(\Omega i)\circ p^3\circ r\simeq p^3\circ(\Omega i)\circ r\simeq p^3
\end{equation*}
Therefore
\begin{equation*}
p^3\simeq (\Omega i)\circ(\Omega\degppp)\circ r \simeq g\circ(\Omega i)\circ r \simeq g.
\end{equation*}
As $g$ is an H-map we conclude that the $p^3$ power map on $SU(p+t-1)$ is an H-map.
\end{proof}

\bibliographystyle{acm}	
\bibliography{Bibfile}

\end{document}